\documentclass[a4paper,11pt]{amsart}

\usepackage{amsfonts}
\usepackage{amsthm}
\usepackage{amssymb}
\usepackage{amsmath}
\usepackage{url, hyperref}
\usepackage{graphicx}
\usepackage{caption}
\usepackage{subcaption}

\newtheorem{theo}{\bf Theorem}[section]
\newtheorem{lemma}{\bf Lemma}[section]
\newtheorem{coro}{\bf Corollary}[section]

\newtheorem{propo}{\bf Proposition}[section]

\newcommand{\spn}{{\rm span}}

\newcommand{\Z}{{\mathbb Z}}
\newcommand{\N}{{\mathbb N}}
\newcommand{\Q}{{\mathbb Q}}
\newcommand{\R}{{\mathbb R}}

\newcommand{\bea}{\begin{eqnarray*}}
\newcommand{\eea}{\end{eqnarray*}}
\newcommand{\be}{\begin{eqnarray}}
\newcommand{\ee}{\end{eqnarray}}

\newcommand{\ve}{\boldsymbol}

\renewcommand{\mod}{\bmod}
\newcommand{\coneS}{{\mathcal S}}

\numberwithin{equation}{section}

\begin{document}

\title[Iterated Chv\'atal-Gomory cuts]
{Iterated Chv\'atal-Gomory Cuts\\ and the Geometry of Numbers}
\author{Iskander Aliev}
\address{School of Mathematics, Cardiff University, Cardiff, Wales, UK}
\email{alievi@cf.ac.uk}

\author{Adam Letchford}
\address{Department of Management Science, Lancaster University, Lancaster, UK}
\email{a.n.letchford@lancaster.ac.uk}

\date{Draft, 19th
June 2013}


\begin{abstract}
Chv\'atal-Gomory cutting planes (CG-cuts for short) are a fundamental tool in
Integer Programming. Given any single CG-cut, one can derive an entire family
of CG-cuts, by `iterating' its multiplier vector modulo one. This leads
naturally to two questions: first, which iterates correspond to the strongest
cuts, and, second, can we find such strong cuts efficiently? We answer the first
question empirically, by showing that one specific approach for selecting the iterate
tends to perform much better than several others. The approach essentially consists
in solving a nonlinear optimization problem over a special lattice associated with the
CG-cut. We then provide a partial answer to the second question, by presenting a
polynomial-time algorithm that yields an iterate that is strong in a certain well-defined
sense. The algorithm is based on results from the algorithmic geometry of numbers.
\end{abstract}



\maketitle

\section{Introduction}

Let ${\ve x} \in \Z^n$ be a vector of integer-constrained decision variables,
and let $A {\ve x} \le {\ve b}$ be a system of linear inequalities, where
$A \in \Z^{m \times n}$ and ${\ve b} \in \Z^m$. A {\em Chv\'atal-Gomory
cutting plane}, or {\em CG-cut}\/ for short, is a linear inequality of the
form
\begin{equation} \label{eq:CGC}
\left( {\ve \lambda}^T  A \right) {\ve x} \le
\left\lfloor {\ve \lambda}^T {\ve b} \right\rfloor,
\end{equation}
for some multiplier vector ${\ve \lambda} \in \R^m_{\ge 0}$ with
${\ve \lambda}^T  A \in \Z$.  (Here, $\lfloor \cdot \rfloor$ denotes rounding
down to the nearest integer. If ${\ve \lambda}^T {\ve b} \in \Z$, we call the
CG-cut (\ref{eq:CGC}) {\em trivial}.)

CG-cuts are so-called because they were derived by Chv\'atal \cite{C73}, based
on earlier work of Gomory \cite{G58, G63}. They form a fundamental family of
cutting planes for {\em Integer Linear Programs}\/ (ILPs); see, e.g.,
\cite{NW88,W98}.

A large number of papers have appeared that use CG-cuts either theoretically
or algorithmically. We survey some of them in Section \ref{se:literature}. One
well-known operation in the literature for creating new CG-cuts from old ones
is to take a multiplier vector ${\ve \lambda}$ and an integer $t$, and create
the new multiplier vector
$t {\ve \lambda} \mod 1:=t {\ve \lambda} - \lfloor t {\ve \lambda} \rfloor$.
(When $\lfloor \cdot \rfloor$ is applied to a vector, each component of the
vector is rounded down.) We call this operation `iterating modulo $1$'.

This leads naturally to two questions: first, which choices for the integer $t$
correspond to strong cuts, and, second, can we find such strong cuts efficiently?
In this paper, we answer the first question empirically, by showing that one
specific approach for selecting $t$ tends to perform much better than several others.
The approach essentially amounts to solving a nonlinear optimization problem over a
special lattice associated with the initial cut. To address the second question, we
first show that for a `typical' cut the covering radius of the associated lattice is small.
This result justifies using the covering radius for estimating the quality of the iterates.
We then provide a partial answer to the second question, by showing the existence
of a polynomial-time algorithm that computes an iterated CG-cut that is strong in a
certain well-defined sense. The algorithm is based on results from the algorithmic geometry
of numbers and computational Diophantine approximations.

The structure of the paper is as follows. The relevant literature is briefly
reviewed in the next section. In Section \ref{se:rules}, we describe several
rules, both known and new, for selecting the integer $t$, and study their
empirical performance. In Section \ref{se:random}, we study the properties of
the iterates for the case in which ${\ve \lambda}$ is random. The polynomial-time
algorithm mentioned above is presented in Section \ref{se:algorithm}. Finally,
some concluding remarks are made in Section~\ref{se:end}.

\section{Literature Review} \label{se:literature}

In this section, we review some relevant papers, introducing some useful
notation and terminology along the way.

\subsection{Gomory fractional cuts}

The original method of Gomory \cite{G58} was designed for ILPs of the form:
\[
\max \left\{ {\ve c}^Tx: \: C {\ve x} = {\ve d}, \:
{\ve x} \in \Z_+^n \right\},
\]
where ${\ve c} \in \Z^n$, $C \in \Z^{p \times n}$ and ${\ve d} \in \Z^p$.
The first step is to solve the {\em Linear Program}\/ (LP)
\[
\max \left\{ {\ve c}^T {\ve x}: \: C {\ve x} = {\ve d}, \:
{\ve x} \in \R_+^n \right\}
\]
by the simplex method. Let ${\ve x}^*$ be the optimal solution to this LP, and
suppose that $x^*_k \notin \Z$ for some $1 \le k \le n$. Then $x_k$ is basic,
and there exists a row of the simplex tableau of the form:
\begin{equation} \label{eq:row}
x_k + \sum_{i \in B} \alpha_i x_i = x_k^*,
\end{equation}
where $B$ is the set of non-basic variables. Rounding down each coefficient to
the nearest integer, we obtain the valid inequality:
\[
x_k + \sum_{i \in B} \lfloor \alpha_i \rfloor x_i \le \lfloor x_k^* \rfloor.
\]
Using the equation (\ref{eq:row}), this inequality can be written as:
\begin{equation} \label{eq:GFC}
\sum_{i \in B} \{ \alpha_i \} x_i \ge \{ x_k^* \},
\end{equation}
where $\{ r \}=r - \lfloor r \rfloor$ is
the fractional part of $r$. The inequality (\ref{eq:GFC})
has come to be known as the {\em Gomory fractional cut}. We will write
GF-cut for short.

Gomory (\cite{G63}, Section 4) pointed out that, by taking integral combinations of the
rows of the simplex tableau, one can create new equations, from which further
GF-cuts can be derived. In this way, he derived a `group' of GF-cuts. He
showed that, unless the original ILP possesses an unusual degree of symmetry,
then the group is {\em cyclic}, which means that the entire group can be
derived by taking integral multiples of one single equation in the tableau.

\subsection{Separation of Chv\'atal-Gomory cuts} \label{sub:lit-sep}

Returning to CG-cuts, define the polyhedron
\begin{equation} \label{P}
P=\left\{ {\ve x} \in \R^n: \:  A {\ve x} \le {\ve b} \right\},
\end{equation}
and let $P_I$ be the convex hull of $P \cap \Z^n$, i.e., the so-called
{\em integral hull}\/ of $P$. Chv\'atal \cite{C73} defined the
{\em elementary closure}\/ of the $P$, denoted by $P'$, as the convex set
that remains after all CG-cuts have been added. Clearly,
$P_I \subseteq P' \subseteq P$. Schrijver \cite{S80} showed that $P'$ is a
polyhedron, or, equivalently, that a finite subset of the CG-cuts dominates
all others.

Now we consider the separation problem for CG-cuts. If $P$ is pointed and
${\ve x}^*$ is a fractional extreme point of $P$, then one can generate a
violated CG-cut via the following four-step procedure: (i) add slack variables
to convert the inequality system $A {\ve x} \le {\ve b}$ into an equation system,
(ii) express ${\ve x}^*$ as a basic feasible solution to that equation system,
(iii) generate a GF-cut, and (iv) convert the GF-cut into a CG-cut by eliminating
slack variables. (For details, see, e.g., Sect.~II.1.3 of \cite{NW88}.) For
general ${\ve x}^*$, however, separation over $P'$ is $NP$-hard (Eisenbrand
\cite{E99}). Fischetti and Lodi \cite{FL07} present an integer programming approach
for separating over $P'$ in practice. Fast separation heuristics have been
presented, for example, in \cite{CF96,CFL00,L02}.

\subsection{Cut strengthening} \label{sub:lit-strengthen}

GF-cuts and CG-cuts may induce facets of $P_I$ in certain cases (see again \cite{CF96,CFL00}).
In general, however, the GF-cuts generated by Gomory's method, or the CG-cuts generated
by existing separation heuristics, can be rather weak. There is a considerable literature on
the derivation of general families of valid linear inequalities which dominate the GF-cuts
and/or CG-cuts (e.g., \cite{CKS90,DG06,G63b,LL02,NW88,NW90}). The drawback of the
inequalities described in those papers is that their coefficients are typically numerically
less stable than those of GF-cuts and CG-cuts. (Recall that CG-cuts have integer coefficients
by definition, and that any GF-cut can be written as a CG-cut.)

An alternative way to address the issue of cut weakness is to develop procedures which take
one or more vectors ${\ve \alpha}\mod 1$ (or, equivalently, one or more multiplier vectors
${\ve \lambda}$), and attempt to construct another vector with more desirable
properties. (Here, ${\ve \alpha}$ is the vector with components $\alpha_i$ from (\ref{eq:row}).)
Here are three examples of such procedures:
\begin{itemize}
\item Gomory (\cite{G63}, Section 5) pointed out that, if $\{ x_i^* \} < 1/2$, then one can
obtain a GF-cut that is at least as strong as the original, by multiplying the
equation (\ref{eq:row}) by the largest positive integer $t$ such that
$1/2 \le t \{ x_i^* \} < 1$.
\item For the same case, Letchford and Lodi \cite{LL02} suggested instead to
multiply the equation (\ref{eq:row}) by $-1$.
\item Ceria {\em et al.}\ \cite{CCD95} gave a heuristic, based on solving systems
of linear congruences, to find a member of the group of GF-cuts with as many zero
left-hand side coefficients as possible.
\end{itemize}
We follow the same approach in this paper, but use more sophisticated algorithmic tools.

We remark that sequences $t {\ve \lambda}\mod 1$ have been investigated in
a completely different context, that of the {\em method of good lattice points}\/
in numerical integration. See, e.g., \cite{Hlawka,Korobov,Sloan}. We remark also
that this is not the first paper to apply tools from the geometry of numbers to
integer programming; see the survey \cite{ESurvey}.

\section{Rules for Finding a Good Iterate} \label{se:rules}

In this section, we examine various rules for finding a good iterated CG-cut, or, equivalently,
for selecting the integer $t$. Throughout this
section, and in the following two, we make an important assumption. Let
${\ve x}^* \in P \setminus P'$ be a fractional point that we wish to separate, and let
${\ve \lambda}$ be an initial multiplier vector. The assumption is that
${\ve \lambda}^T ({\ve b} - A{\ve x}^*) = {\ve 0}$, i.e., that all
inequalities with a positive multiplier have zero slack at ${\ve x^*}$. This assumption
holds, for example, when ${\ve x}^*$ is an extreme point of $P$ and the CG-cut has
been generated by the four-step procedure mentioned in Subsection \ref{sub:lit-sep}.
It also holds when the CG-cut has been generated using the separation heuristics in
\cite{CFL00,L02}. It has the important implication that, regardless of the integer $t$,
every non-trivial iterated CG-cut will be violated by ${\ve x}^*$.

In the following three subsections, we present some useful notation, describe six specific
rules for selecting an iterate, and present some preliminary computational results.

\subsection{Some useful notation}

It follows from results in Schrijver \cite{S80} that we can assume, without loss of generality,
that ${\ve \lambda}$ is rational.
Furthermore, the CG-cut $\left( {\ve \lambda}^T  A \right) {\ve x} \le
\left\lfloor {\ve \lambda}^T {\ve b} \right\rfloor$ is implied by $A{\ve x}\le {\ve b}$ and the CG-cut $\left( ( {\ve \lambda} \mod 1)^T  A \right) {\ve x} \le
\left\lfloor (  {\ve \lambda} \mod 1)^T {\ve b} \right\rfloor$. Thus we may also assume that ${\ve \lambda}\in [0,1)^m$.

Therefore we can write
\be
{\ve \lambda} =
\left(\frac{p_1}{q}, \frac{p_2}{q}, \ldots, \frac{p_m}{q}\right)^T\,,
\label{standard_lambda}
\ee
where $q$ is a positive integer and $p_1, p_2,\ldots, p_m$ are non-negative
integers with $\gcd(p_1, p_2,\ldots, p_m,q)=1$. Then, for any integer
$1 \le t < q$, the inequality
\begin{equation} \label{eq:CG-t}
\left( ( t {\ve \lambda} \mod 1)^T  A \right) {\ve x} \le
\left\lfloor ( t {\ve \lambda} \mod 1)^T {\ve b} \right\rfloor
\end{equation}
is a (possibly trivial) iterated CG-cut.

The family of iterated CG-cuts formed in this way is analogous to the group of GF-cuts
described by Gomory, or, more precisely, to the subgroup of GF-cuts that can be derived
by taking integer multiples of one single row of the tableau. Note that $q$ can be
exponentially large, and so can the family of iterated CG-cuts.

At this point, it is helpful to define the {\em slack vector}\/
${\ve s} = {\ve b} - A{\ve x}$ and the {\em rounding effect}\/
$\nu = \left\{ {\ve \lambda}^T {\ve b} \right\}$. Then, the iterated CG-cut
(\ref{eq:CG-t}) can be written in the alternative form:
\begin{equation} \label{eq:CG-t-slack}
(t {\ve \lambda} \mod 1)^T {\ve s} \ge \left\{ t \nu \right\}.
\end{equation}
Now, since we are assuming that ${\ve \lambda}^T {\ve s} = {\ve 0}$ at ${\ve x}^*$, the
left-hand side of (\ref{eq:CG-t-slack}) at ${\ve x}^*$ will be zero. This means that,
provided that an iterated CG-cut is not trivial, it will be violated by ${\ve x}^*$.

\subsection{Six specific rules}

Now we consider how to select the integer $t$. A trivial strategy, which we call
{\em Strategy 0}, is to select $t=1$. As mentioned in Subsection \ref{sub:lit-strengthen},
Gomory \cite{G63} suggested to set $t=1$ if $\nu < 1/2$, but to the largest integer such
that $t \nu < 1$ otherwise; and Letchford and Lodi \cite{LL02} suggested to set $t=1$ if
$\nu < 1/2$, but to $-1$ otherwise. We will call these approaches {\em Strategy 1} and
{\em Strategy 2}, respectively. Another approach, that we call {\em Strategy 3},
is to select an integer $t$ such that the right-hand side of (\ref{eq:CG-t-slack}) is
maximised.

The previous three strategies are concerned only with making the right-hand side of
(\ref{eq:CG-t-slack}) (rounding effect) large. It is also desirable for the left-hand side
to have small norm. In this paper we propose to optimize these two quantities {\em simultaneously}.
We consider two strategies, {\em multiplicative} and {\em additive}, to ensure that the norm
of the multiplier vector is small, but the rounding effect is large.

The multiplicative strategy attempts to minimise the ratio
\[
||t {\ve \lambda} \mod 1||/\{t \nu \}
\]
over all iterations with positive rounding effect $\{t\nu\}$. Here $||\cdot||$ denotes the Euclidean norm.
That is, we are solving the
following optimization problem:
\be
\min \left\{ ||t {\ve \lambda} \mod 1||/\{t \nu \}: \;
t=1,\ldots,q-1, \; \{t \nu\}>0 \right\}\,.
\label{SM}
\ee
We will call this {\em Strategy 4}. Unfortunately,
the complexity of this problem is unknown. We conjecture that it is $NP$-hard.

Let us now construct the augmented vector
\[
{\ve \nu} = ( \lambda_1, \ldots, \lambda_m, \nu )^T\,
\]
and put for ${\ve x}=(x_1,\ldots, x_{d-1}, x_d)$
\bea
N({\ve x})=||(x_1,\ldots, x_{d-1}, 1-x_d)||\,.
\eea

The additive strategy attempts to find a vector ${\ve \xi}=t {\ve \nu} \mod 1$ with minimum value
$N({\ve x})$ and positive last entry $\xi_d=\{t \nu\}$, which represents the rounding effect of the
iterated cut. That is, we are solving the following optimization problem:
\be
\min \left\{ N(t{\ve \nu}\mod 1): \;
t=1,\ldots,q-1, \; \{t \nu\}>0 \right\} .
\label{SCP}
\ee
We call this {\em Strategy 5}. We conjecture that this problem too is $NP$-hard. In Section
\ref{se:algorithm}, we show that both problems (\ref{SM}) and (\ref{SCP}) can be solved approximately
in polynomial time.

Note that the new Strategies 4 and 5 (as well as the Strategies 0--3) do not depend on the objective function. Finding an effective strategy that employs the parameters of the objective function is a
topic for future research.

\begin{table}[]
\caption{Computational results}
\centering
\small
\begin{tabular}{ ccrrrrrr }
\hline \noalign{\smallskip}
$m$ & $n$ & $A_0(m,n)$ & $A_1(m,n)$ & $A_2(m,n)$ & $A_3(m,n)$ & $A_4(m,n)$ & $A_5(m,n)$ \\ 
\noalign{\smallskip} \hline \noalign{\smallskip}
   & 10 & 19.29 & 35.46  & 30.98 & 34.99 & 45.94 & 45.02 \\

5  & 20 & 18.89 & 29.14 & 23.56 & 33.67 & 36.00 & 40.84 \\

   & 30 & 14.09 & 21.76 & 17.23 & 19.99 & 21.20 & 29.74 \\[1.2ex]

   & 10 & 4.52 & 5.95 & 4.66 & 6.68 & 13.22 & 10.46 \\

10 & 20 & 3.14 & 5.90 & 4.97 & 7.13 & 11.37 & 8.90 \\

   & 30 & 3.85 & 6.63 & 4.93 & 6.28 & 10.49 & 9.44 \\[1.2ex]

   & 10 & 5.89 & 9.02 & 7.53 & 10.02 & 15.92 & 16.39 \\

15 & 20 & 1.86  & 2.94 & 2.51 & 3.68 & 14.16 & 12.15 \\

   & 30 & 2.87  & 4.04 & 3.16 & 3.28 & 10.25 & 10.00 \\[1.2ex]
\noalign{\smallskip} \hline \noalign{\smallskip}
\multicolumn{2}{c}{$A_s:$} & 8.27 & 13.43 & 11.06 & 13.97 & 19.83 & 20.33\\
\noalign{\smallskip} \hline
\end{tabular}
\label{table:gaps}
\end{table}

\subsection{Preliminary computational results}

In order to gain some insight into the performance of the six strategies mentioned in the
previous subsection, we performed some computational experiments on some small ILPs. We began
by creating 45 random ILPs of the form
\[
\max \left\{ {\ve c}^T{\ve x}: \: A{\ve x} \le {\ve b}, \:  {\ve x} \in \Z_+^n \right\},
\]
where ${\ve c} \in \Z_+^n$, $A \in \Z_+^{m \times n}$ and ${\ve b} \in \Z_+^m$. (Note
that instances of this form are guaranteed to be feasible, since the origin
is feasible.) For any pair $(n,m)$ with
$m \in \{5, 10, 15\}$ and $n \in \{10, 20, 30\}$, 5 such instances $(m,n,k)$, $k\in \{1,\ldots,5\}$
were constructed. The $c_i$ were random integers distributed uniformly between 1 and 5.
The $A_{ij}$ were random integers with a 50\% chance of being distributed uniformly between 1 and 5,
but a 50\% chance of being zero. This was to mimic the sparsity that is usually found in real-life
ILPs. (If any column of $A$ had fewer than two non-zeroes, the column was discarded and another one
generated. This is to ensure boundedness.) The $b_j$ were set to
$\left\lceil \frac12 \sum_{i=1}^n A_{ij} \right\rceil$.

For each instance $(m,n,k)$, the LP relaxation was solved to optimality and the optimal simplex
tableau computed using exact rational arithmetic. (To avoid numerical problems, instances for
which the determinant $D$ of the basis matrix exceeded $2\cdot 10^6$ were discarded. The desire to
keep $D$ small also motivated the above restrictions on the coefficients.) Then, for each variable
taking a fractional value in the LP solution, whether a structural variable or a slack variable,
a GF-cut was generated and converted into a CG-cut. At the end, for the instance $(m,n,k)$ and for
each of the strategies $s \in \{0,\ldots,5\}$, we stored the average $A_{s}(m,n,k)$, over all
considered CG-cuts, of the percentage of the integrality gap closed by a CG-cut.

In Table \ref{table:gaps} below, we compare all six strategies. For each value of $(n,m)$ and for
each of the strategies $s \in \{0,\ldots,5\}$, we report the average
$A_s(m,n)=(1/5)\sum_{k=1}^{5}A_s(m,n,k)$. In the last row of the table, the numbers
$A_s=(1/9)\sum_{m,n} A_s(m,n)$ are the averages of $A_s(m,n,k)$ over all computed instances. 

The computational results show that both Strategies 4 and 5 close significantly more
of the integrality gap than the other four strategies. This indicates that the rounding
effect and the norm of the multiplier vector should be simultaneously optimized for
generating strong CG-cuts. To gain an insight on the theoretical aspects of this problem,
we study in the next section the behavior of the iterated cuts for a randomly chosen
augmented vector ${\ve \nu}$.

\section{Behaviour of the Iterates for a Random Vector}
\label{se:random}

As illustrated by Figure \ref{fig:distributions}, the values of the minima in
(\ref{SM}) and (\ref{SCP}) may vary significantly from one vector ${\ve \nu}$ to another, even
for a fixed $q$. Intuitively, the chance of obtaining a good iterate is higher
if the iterates are `spread' reasonably uniformly over the hypercube, as in
cases C and D. This led us to examine the behaviour of the iterates for `typical'
vectors ${\ve \nu}$.

\begin{figure}[]
        \begin{subfigure}[b]{0.4\textwidth}
                \centering
                \includegraphics[width=\textwidth]{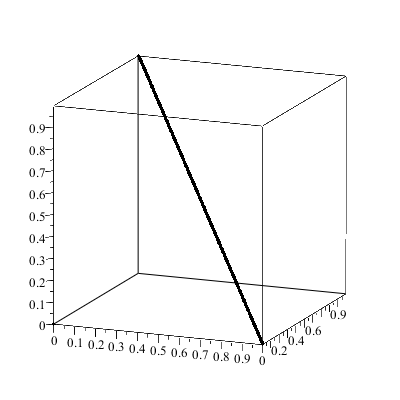}
                \caption{}
                \label{fig:distr1}
        \end{subfigure}%
        ~ 
        \begin{subfigure}[b]{0.4\textwidth}
                \centering
                \includegraphics[width=\textwidth]{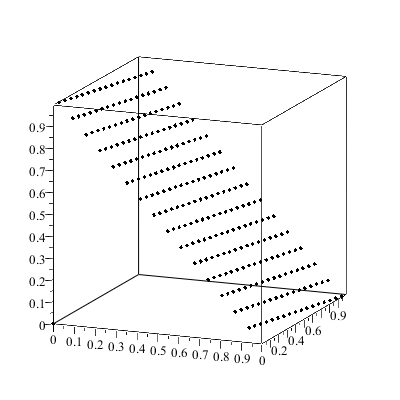}
                \caption{}
                \label{fig:distr2}
        \end{subfigure}%

        \begin{subfigure}[b]{0.4\textwidth}
                \centering
                \includegraphics[width=\textwidth]{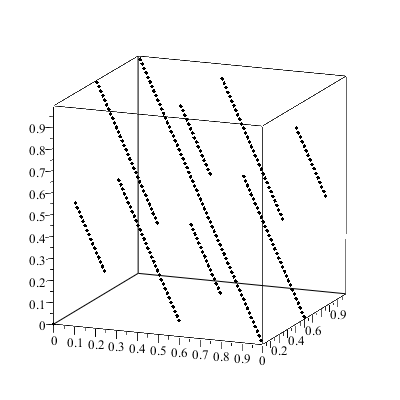}
                \caption{}
                \label{fig:distr3}
        \end{subfigure}
        ~ 
        \begin{subfigure}[b]{0.4\textwidth}
                \centering
                \includegraphics[width=\textwidth]{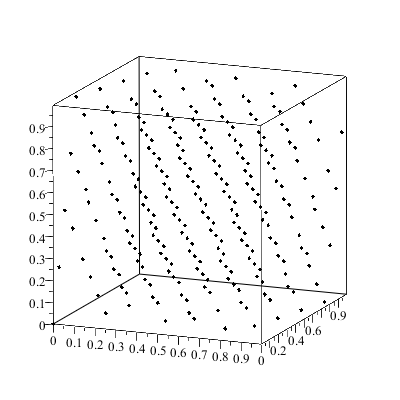}
                \caption{}
                \label{fig:distr4}
        \end{subfigure}%
        \caption{\scriptsize Examples of sequences ${\ve \nu}\mod 1$ with $q=256$: (A) ${\ve \nu}=(1/256,255/256,255/256)$; (B) ${\ve \nu}=(1/256,15/256,255/256)$; (C) ${\ve \nu}=(1/256,63/256,127/256)$ and (D) ${\ve \nu}=(1/256,15/256,63/256)$.}\label{fig:distributions}
\end{figure}

To formulate the obtained results, we need to introduce the following notation.
Given a matrix $B \in \R^{d \times l}$ with linearly independent column vectors
${\ve b}_1, \ldots, {\ve b}_l \in \R^d$, the set
\bea
L({\ve b}_1, \ldots, {\ve b}_l)=\{u_1{\ve b}_1+ \cdots+ u_l{\ve b}_l: u_1, \ldots, u_l\in \Z\}\,
\eea
is called a {\em lattice}\/ of {\em rank} (or {\em dimension})\/ $l$ with {\em basis}\/ ${\ve b}_1, \ldots, {\ve b}_l$
and {\em determinant}
\bea
\det(L({\ve b}_1, \ldots, {\ve b}_l))=\sqrt{\det(B^T B)}\,.
\eea
For a comprehensive and extensive
survey on lattices and Minkowski's geometry of numbers we refer the reader to the book of Gruber and Lekkerkerker \cite{{GrLek}}.

Given lattice $L$, we will denote by $L^*$ its {\em dual} lattice, that is
\bea
L^*=\{{\ve y}\in \spn_{\R}(L): {\ve y}^T{\ve x}\in \Z\,\;\mbox{for all}\; {\ve x}\in L\}\,.
\eea

Let $B^d({\ve x}, r)$ denote a $d$-dimensional ball of radius $r$ centered at ${\ve x}$. Given any $l$-dimensional lattice $\Lambda\subset \R^d$ we also denote by $\lambda_i=\lambda_i(\Lambda)$
its $i$th successive minimum
\bea
\lambda_i=\min\{r>0: \dim \spn_{\R}(B^d({\ve 0}, r)\cap \Lambda)\ge i\}\,,\;1\le i\le l\,.
\eea

Recall that the {\em inhomogeneous minimum}\/ of a set $S \subset \spn_\R(L)$ with respect
to a lattice $L$ is defined as
\bea \mu(S,L)= \inf\{\sigma>0: L + \sigma S=\spn_\R(L)\}\,. \eea
The {\em covering radius}\/ $\tau(L)$ of a lattice $L$ is the inhomogeneous minimum
of the unit ball $B$ in $\spn_\R(L)$ with respect to $L$,
\bea
\tau(L)=\mu(B,L)\,.
\eea

Let also $\ll_d$ (resp. $\gg_d$) denote the Vinogradov symbol with the constant
depending on $d$ only. The notation $\gg\ll_d$ is interpreted as both $\ll_d$ and $\gg_d$ hold.

We will first study the `typical' behaviour of the iterates, for a random vector ${\ve \nu}$
sampled from a certain natural distribution. In particular, we show that the covering radius
of a lattice associated with ${\ve \nu}$ is relatively small on average. This is important
because the quality of the approximation algorithm presented in Section \ref{se:algorithm}
will be defined in terms of the covering radius. (Of course, a multiplier vector obtained in
a real cutting-plane algorithm will not be truly random. Nevertheless, the insights gained in
this section will be useful for what follows.)

In more detail, we study in this section the behavior of the points $t {\ve \nu} \mod 1$ for a
random vector ${\ve \nu}$ uniformly chosen from the set of rational vectors of the form
\be \begin{split} {\ve \nu}=\left(\frac{p_1}{q}, \frac{p_2}{q}, \ldots, \frac{p_d}{q}\right)^T\,,\; p_1, p_2,\ldots, p_d,q\in \Z_{>0}\,,\\ \max_{1\le i\le d} p_i<q\,,\;\gcd(p_1, p_2,\ldots, p_d,q)=1 \end{split}\label{nu_via_p} \ee
that have denominator $q \le T$, for some $T \ge 1$. Our aim is to understand how well the
points $t {\ve \nu}\mod 1$ are distributed `on average'.

The iterates $t{\ve \nu} \mod 1$ can be naturally embedded in the lattice
\be
L_{\ve \nu}=\{{\ve z}+(t{\ve \nu} \mod 1): {\ve z}\in \Z^d\,,\; t=1,\ldots,q-1 \}\,.
\label{z+lambda}
\ee
Equivalently, $L_{\ve \nu} = \Z^d + \Z {\ve \nu}$.
This observation allows us to use results from Minkowski's geometry of numbers
and, via the transference principle (see, e.g., \cite{Banaszczyk}), Schmidt's
theorems \cite{Schmidt} on the distribution of integer sublattices.

The first result of this paper aims to understand the `typical' behavior of the
covering radius $\tau(L_{\ve \nu})$ for ${\ve \nu}$ of the form (\ref{nu_via_p}) with common denominator $q \le T$.
Note that for any dimension $d$ and any common denominator $q>1$ there exist vectors $\nu$ such that the covering radius $\tau(L_{\ve \nu})$ is relatively large.  For instance, it is easy to see that $\tau(L_{(1/q, \ldots, 1/q)})>1/4$ for any integer $q>1$. In what follows, we will show that for a `typical' vector $\nu$ the covering radius $\tau(L_{\ve \nu})$ has the order $q^{-1/d}$.

For technical reasons it is convenient to replace the rationals with bounded
denominators by the primitive integer vectors in a bounded domain.
Let $\widehat\N^{d+1}$ be the set of integer vectors in $\R^{d+1}$ with positive
co-prime coefficients, and let
\[
{\mathcal D}_{d+1} = \left\{ (x_1, \ldots, x_{d}, x_{d+1}) \in \R^{d+1}_{\ge 0}:
\max_{j=1,\ldots,d} x_j < x_{d+1} \le 1 \right\}.
\]
Then for $T \ge 1$,
the elements ${\ve a}=(p_1,\ldots,p_d,q)$ of the set
$ \widehat\N^{d+1}\cap T {\mathcal D}_{d+1}$ will correspond to the rational vectors
${\ve \nu}$ of the form (\ref{nu_via_p}) and the common denominator $q\le T$.
Since $ L_{\ve \nu}$ is uniquely defined by the integer vector ${\ve a}=(p_1,\ldots,p_d,q)$,
we will also denote the lattice $L_{\ve \nu}$ by $L_{\ve a}$.

For any $T \in \R_{\ge 1}$ and $R \in \R_{>0}$, we define the quantity
\bea
P_{d}(T,R)=\frac{1}{\#(\widehat\N^{d+1}\cap T{\mathcal D}_{d+1})}\#\left\{{\ve a}\in\widehat\N^{d+1}\cap T{\mathcal D}_{d+1}:
\tau(L_{\ve a})a_{d+1}^{1/d}>R \right\}\,.
\eea
Roughly speaking, $P_{d}(T,R)$ is the probability of uniformly picking up a
rational vector ${\ve \nu}$ of the form (\ref{nu_via_p}) with denominator $q \le T$,
such that the iterations $t{\ve \nu} \mod 1$ are relatively badly distributed in
$[0,1)^d$ or, more precisely, such that the covering radius of the lattice $L_{\ve \nu}$ is
bigger than $R q^{-1/d}$.
\begin{theo}
Let $d\ge 2$. Then
\be
P_{d}(T,R)
\ll_{d} R^{-d}\,,
\label{bound_for_P}
\ee
uniformly over all $T\ge 1$  and all $R>0$. Furthermore,
\be
P_d(T,R)=0\,\;\mbox{whenever}\; R>\frac{\sqrt{d}}{2}T^{1/d}\,.
\label{zero}
\ee
\label{probability}
\end{theo}

A celebrated result of Kannan \cite{Kannan} implies that the {\em Frobenius number}  associated with an integer vector ${\ve a}\in \widehat\N^{d+1}$
can be estimated in terms of the covering radius of the dual lattice $L_{\ve a}^*$. (For more details  we refer the reader to the book of Ramirez Alfonsin \cite{Alf}.)
The following proof of Theorem \ref{probability} is based on  a recent far-reaching refinement
due to Str\"ombergsson \cite{Str} of the approach used in \cite{AlievHenk} and
\cite{AHH} for estimating the expected value of Frobenius numbers, combined with the Banaszczyk
transference theorem \cite{Banaszczyk}. The approach is
built on results from the Minkowski's geometry of numbers (see e. g. \cite{peterbible},
\cite{GrLek}) and results on the distribution of integer lattices obtained by Schmidt
in \cite{Schmidt}.

\subsection*{Proof of Theorem \ref{probability}}

Observe that $\Z^d$ is a sublattice of $L_{\ve a}$ and hence
\be
\tau(L_{\ve a})\le \tau(\Z^d)=\frac{\sqrt{d}}{2}\,.
\label{just_a_sublattice}
\ee
Note also that for all ${\ve a}\in \widehat\N^{d+1}\cap T{\mathcal D}_{d+1}$ we have $a_{d+1}\le T$. Hence the inequality (\ref{just_a_sublattice}) implies (\ref{zero}).

Let us now prove that the inequality (\ref{bound_for_P}) holds. For a subset $Y\subset  \R^{d+1}$ we denote by $\pi_{d+1}(Y)$ the orthogonal projection of $Y$ onto the coordinate hyperplane $x_{d+1}=0$; we view $\pi_{d+1}(Y)$ as a subset of $\R^d$.
Given ${\ve a}\in \widehat\N^{d+1}$, we define the lattice
\bea
\Lambda_{\ve a}=
\{{\ve x}\in \Z^{d+1}: {\ve x}^T{\ve a}=0\}\,
\eea
and set $M_{\ve a}=\pi_{d+1}(\Lambda_{\ve a})$. Then $M_{\ve a}$ is a sublattice of $\Z^d$ of determinant $\det(M_{\ve a})=a_{d+1}$ (see e. g. \cite{AG}, Section 2)
%
It is well-known that $M_{\ve a}=L_{\ve a}^*$ (see e.g. \cite{AG}).

%

By Banaszczyk transference theorem \cite{Banaszczyk}, we have
\bea
 \lambda_1(M_{\ve a})\le \frac{d}{2\tau(L_{\ve a})}\,.
\eea
Since $M_{\ve a}$ embedded in $\R^{d+1}$ is the orthogonal projection of $\Lambda_{\ve a}$ on the coordinate hyperplane $x_{d+1}=0$ and $a_{d+1}=\max_i a_i$, we have $\lambda_1(\Lambda_{\ve a})\le \sqrt{d+1}\lambda_1(M_{\ve a})$ and, consequently,
\be
 \lambda_1(\Lambda_{\ve a})\le \frac{d\sqrt{d+1}}{2\tau(L_{\ve a})}\,.
\label{lambda_tau}
\ee

In the rest of this subsection we modify the proof of Theorem 3 in \cite{Str} for
our case. Roughly speaking, the main difference is that, due to the transference
principle reflected in the inequality (\ref{lambda_tau}), we need to work with the
first successive minimum $\lambda_1(\Lambda_{\ve a})$, whilst in the case of the
Frobenius number the last successive minimum $\lambda_{d}(\Lambda_{\ve a})$ plays
the major role.

Note first that $\#(\widehat\N^{d+1}\cap T{\mathcal D}_{d+1})\gg \ll_d T^{d+1}$ uniformly over all $T\ge 1$ and that $\det(\Lambda_{\ve a})=||{\ve a}||\ge  a_{d+1}$. Therefore
\be
\begin{split}
P_d(T,R)\ll_d   T^{-(d+1)} \times \\ \times\#\left\{\Lambda\in {\mathcal L}_d:  \det(\Lambda)\le \sqrt{d+1} T,  \lambda_1(\Lambda)<   \frac{d\sqrt{d+1}\det(\Lambda)^{1/d}}{2R}              \right\}\,,
\end{split}
\label{right_hand_side}
\ee
where ${\mathcal L}_d$ is the set of all $d$-dimensional sublattices of $\Z^{d+1}$.

Let
\bea
\rho_j(\Lambda)= \lambda_{j+1}(\Lambda)/\lambda_j(\Lambda)\,,\;\; j=1,\ldots, d-1\,.
\eea
For any ${\ve r}=(r_1, \ldots, r_{d-1})\in \R^{d-1}_{\ge 1}$ we set
\bea
{\mathcal L}_d({\ve r})=\{\Lambda\in {\mathcal L}_d: \rho_j(\Lambda)\ge r_j\,,\;1\le j\le d-1\}\,.
\eea
Let also $X_d$ be the set of all lattices $L\subset \R^d$ of determinant one and $\mu_d$ be Siegel's measure (see \cite{Siegel}) on $X_d$, normalized to be a probability measure.
The main ingredient of the proof is the following result.

\begin{theo}[Schmidt \cite{Schmidt}]
For any ${\ve r}\in \R^{d-1}_{\ge 1}$ and $T>0$ we have
\be
\begin{split}
\#\{\Lambda\in {\mathcal L}_d({\ve r}): \det(\Lambda)\le T\}=\frac{\pi^{\frac{d+1}{2}}}{2\Gamma\left(1+\frac{d+1}{2}\right)}
\left(\prod_{j=2}^d\zeta(j)\right)\times \\
\times\mu_d\left(\{L\in X_d: \rho_j(L) \ge r_j\,,\; 1\le j\le d-1 \}\right)T^{d+1}\\
+O_{d}\left(\left(\prod_{j=1}^{d-1}r_j^{-(j-\frac{1}{d})(d-j)}\right)T^{d+1-\frac{1}{d}}\right)\,.
\end{split}
\label{Schmidt_asympt}
\ee
Furthermore,
\be
\mu_d(\{L\in X_d: \rho_j(L)\ge r_j\,,\;1\le j\le d-1\})\gg \ll_d\prod_{j=1}^{d-1}r_j^{-j(d-j)}\,.
\label{Schmidt_asympt_main}
\ee

\end{theo}

From the above theorem we get the upper bound

\be
\begin{split}
\#\left\{\Lambda\in{\mathcal L}_d({\ve r}): \det(\Lambda)\le T\right\}\ll_dT^{d+1}\times\\
\times\prod_{j=1}^{d-1}r_j^{-j(d-j)}\left(1+T^{-\frac{1}{d}}\prod_{j=1}^{d-1}r_j^{\frac{1}{d}(d-j)}\right)\,.
\end{split}
\label{upper_bound_Schmidt}
\ee

By Minkowski's Second theorem, for any $d$-dimensional lattice $\Lambda$ we have
\be
\lambda_1(\Lambda)^d=\frac{\prod_{j=1}^d\lambda_j(\Lambda)}{\prod_{j=1}^{d-1} \rho_j(\Lambda)^{d-j}}\gg \ll_d \frac{\det(\Lambda)}{\prod_{j=1}^{d-1} \rho_j(\Lambda)^{d-j}}\,.
\label{lambda_via_rho}
\ee

Thus there exists a constant $c_1=c_1(d)>0$  such that for any $d$-dimensional lattice $\Lambda$ and  any $R>0$,
we have
\be
 \lambda_1(\Lambda)<   \frac{d\sqrt{d+1}\det(\Lambda)^{1/d}}{2R} \Rightarrow \prod_{j=1}^{d-1} \rho_j(\Lambda)^{d-j} >c_1 R^d\,.
\label{tau_implies_rho}
\ee
Assume without loss of generality $R>ec_1^{-\frac{1}{d}}$ (the inequality (\ref{bound_for_P}) is trivial when $R \ll 1$ as $P_{d}(T,R)\le 1$), put
\be
B=\lfloor \log(c_1R^d)-d\rfloor\in \Z_{\ge 0}\,,
\ee
and denote
\be
\begin{split}
{\mathcal R}(d, R)=\{(e^{b_1/(d-1)}, e^{b_2/(d-2)}, \ldots, e^{b_{d-2}/2}, e^{b_{d-1}}):\\ {\ve b}\in \Z^{d-1}_{\ge 0}, \sum_{j=1}^{d-1}b_j=B\}\,.
\end{split}
\ee
If $\Lambda$ is an $d$-dimensional lattice with $\prod_{j=1}^{d-1} \rho_j(\Lambda)^{d-j} >c_1 R^d$, then for $b_j=\lfloor (d-j)\log \rho_j(\Lambda)\rfloor$
we have
\be
\begin{split}
\sum_{j=1}^{d-1}b_j> \sum_{j=1}^{d-1}((d-j)\log \rho_j(\Lambda)-1)>\log(c_1R^d)-(d-1)\\
>\log(c_1R^d)-d\ge B\,.
\end{split}
\ee
Thus we can decrease some of the numbers $b_j$'s so as to make $\sum_{j=1}^{d-1}b_j=B$, while keeping ${\ve b}=(b_1, \ldots, b_{d-1})\in \Z^{d-1}_{\ge 0}$. The new vector
${\ve b}$ still satisfies $b_j\le (d-j)\log \rho_j(\Lambda)$ for each $j$, that is $\rho_j(\Lambda)\ge e^{b_j/(d-j)}$. Therefore, for any $d$-dimensional lattice $\Lambda$
with $\prod_{j=1}^{d-1} \rho_j(\Lambda)^{d-j} >c_1 R^d$, there exists some ${\ve r}\in {\mathcal R}(d, R)$ such that $r_j\le \rho_j(\Lambda)$ for all $j$.

By (\ref{tau_implies_rho}), the set in the right hand side of (\ref{right_hand_side}) is contained in the union of ${\mathcal L}_d({\ve r})$ over all ${\ve r}\in {\mathcal R}(d, R)$.
Hence we have for all $T\ge 1$ and all $R\ge ec_1^{\frac{1}{d}}$,
\be
\begin{split}
P_{d}(T,R)\ll_d T^{-(d+1)}\times \\ \times \sum_{{\ve r}\in {\mathcal R}(d, R)} \#\left\{\Lambda\in {\mathcal L}_d({\ve r}): \det(\Lambda) \le \sqrt{d+1} T\right\}\,.
\end{split}
\label{final_right_hand_side}
\ee

By (\ref{upper_bound_Schmidt}),
\be
\begin{split}
P_d(T,R)\ll_d \sum_{\scriptsize \begin{array}{c}{\ve b}\in \Z^{d-1}_{\ge 0}\\ b_1+\ldots+b_{d-1}=B\end{array} } \exp\left\{-\sum_{j=1}^{d-1}j b_j\right\}\\
+T^{-\frac{1}{d}}\sum_{\scriptsize \begin{array}{c}{\ve b}\in \Z^{d-1}_{\ge 0}\\ b_1+\ldots+b_{d-1}=B\end{array} } \exp\left\{-\sum_{j=1}^{d-1} \left(j-\frac{1}{d}\right)b_j\right\}\,.
\end{split}
\label{right_exponents}
\ee

If $d=2$ we get
\bea
P_2(T,R)\ll R^{-2}+T^{-\frac{1}{2}} R^{-1}.
\eea
If $R\le \frac{\sqrt{2}}{2}T^{1/2}$ then this implies $P_2(T,R)\ll R^{-2}$. On the other hand, if $R> \frac{\sqrt{2}}{2}T^{1/2}$ then $P_2(T,R)=0$ by (\ref{zero}).

Let us now assume $d\ge 3$. Observe that for any ${\ve b}\in \Z_{\ge 0}^{d-1}$ with $b_1+\ldots+b_{d-1}=B$ and
$b_2+\ldots+b_{d-1}=s$, we have
\bea
\sum_{j=1}^{d-1}j b_j\ge B+s
\eea
and
\bea
\sum_{j=1}^{d-1} \left(j-\frac{1}{d}\right)b_j \ge \left(1-\frac{1}{d}\right)B+s\,.
\eea
Next, for $s\in\{0,1,\ldots, B\}$ there are exactly ${{s+d-3}\choose{d-3}}$ vectors ${\ve b}\in\Z^{d-1}_{\ge 0}$
with $b_1+\ldots+b_{d-1}=B$ and $b_2+\ldots+b_{d-1}=s$. Therefore
\bea
\begin{split}
P_d(T,R)\ll_d \sum_{s=0}^B {{s+d-3}\choose{d-3}} e^{-B-s}\\+ T^{-\frac{1}{d}}\sum_{s=0}^B {{s+d-3}\choose{d-3}}
e^{-\left(1-\frac{1}{d}\right)B-s}\ll_d e^{-B}+T^{-\frac{1}{d}}e^{-\left(1-\frac{1}{d}\right)B}\\
\ll_d R^{-d}(1+ T^{-\frac{1}{d}} R)\,.
\end{split}
\eea

If $R\le \frac{\sqrt{d}}{2}T^{1/d}$ then this implies $P_d(T,R)\ll R^{-d}$. On the other hand, if $R> \frac{\sqrt{d}}{2}T^{1/d}$ then $P_d(T,R)=0$ by (\ref{zero}).
The proof is complete.

\section{The Approximation Algorithm} \label{se:algorithm}

We will assume for this section that $d \ge 2$.  Theorem \ref{probability} shows that the quantity
$1/{q}^{1/d}$ is a good predictor for the covering radius of the lattice $L_{\ve \nu}$.
Let $\coneS=[0, +\infty)^{d-1}\times (-\infty, 1)$. The following result states the existence of a
polynomial-time algorithm which computes a point of the set $L_{\ve \nu}\cap \coneS$
in a certain ball of radius bounded in terms of $\tau(L_{\ve \nu})$. The obtained bound will be used to
estimate the quality of polynomial-time approximations for the multiplicative and additive strategies
(i.e., Strategies 4 and 5) introduced in Section \ref{se:rules}.

For $r\in \R$ set
\bea
{\ve c}(r)=(r, \ldots, r, 1-r)\in \R^d\,.
\eea

\begin{theo}
There is a  polynomial time algorithm which, given a rational vector ${\ve \nu}$ of the form
(\ref{nu_via_p}) and any rational $\epsilon\in (0,1)$, finds a point ${\ve \xi}\in L_{\ve \nu}\cap \coneS$, such that
\be
{\ve \xi}\in  B({\ve c}(r), 2^{d/2}\tau(L_{\ve \nu}))\,\;\mbox{with}\;0<r\le 2^{d/2}\tau(L_{\ve \nu})+\epsilon\,.
\label{alg_bound_for_xi}
\ee
\label{length_bound}
\end{theo}
The proof is constructive. We present the polynomial time algorithm in Section \ref{length_bound_section}.

\subsection{Proof of Theorem \ref{length_bound}}
\label{length_bound_section}.

We need to find in polynomial time a point of the set $L_{\nu}\cap \coneS$ in a ball $B^d({\ve c}(r), r)$.   The main challenge of the proof is to choose the radius  $r\ll_d \tau(L_{\ve \nu})$ as small as possible.  Note that computing the covering radius of a lattice is conjectured in \cite{M02} to be NP-hard (see also \cite{Haviv}, \cite{Guruswami} and \cite{Achill}). The Banaszczyk transference theorem \cite{Banaszczyk}
gives the estimate
\bea
\tau(L_{\ve \nu}) \le \frac{d}{2\lambda_1(L_{\ve \nu}^*)}\,,
\eea
which allows to approximate $\tau(L_{\ve \nu})$ in polynomial time within the factor $d 2^{d/2-1}$ using the celebrated LLL algorithm  \cite{LLL}. The approximation can be then used for computing a relatively small radius $r$.

In this paper we use a slightly different approach. We will choose a suitable
radius $r$ by combining binary search in a certain interval with Babai's
{\em nearest plane algorithm}. The nearest plane algorithm finds in polynomial time an
approximation to a solution of the {\em closest vector problem}. The quality of
the approximation is given by the following result.

\begin{theo}[Babai \cite{Babaika}]
Let $L$ be a lattice of rank $d$ in $\Q^{d}$. Given any basis of $L$ and any ${\ve c}\in \Q^{d}$ as input, the nearest plane
algorithm computes a vector ${\ve x}\in L$ such that
\be
||{\ve x}-{\ve c}||\le 2^{d/2} \min_{{\ve y}\in L} ||{\ve y}-{\ve c}||\,.
\label{Babai_factor}
\ee
\label{Babai}
\end{theo}
Babai's nearest plane algorithm is based on using the LLL algorithm and, in fact,
makes use also of the transference principle, via Gram-Schmidt orthogonalization.
Note also that the approximation factor $2^{d/2}$ in (\ref{Babai_factor}) can be
replaced by $2^{O(d(\log\log d)^2/\log d)}$ by applying the algorithm of Schnorr
\cite{Schnorr}.

We shall now give a high level description of a polynomial-time algorithm that satisfies conditions stated in Theorem \ref{length_bound}. Given rational ${\ve \nu}$ of the form
(\ref{nu_via_p}), we first
compute a basis ${\ve u}_1, {\ve u}_2, \ldots, {\ve u}_{d}$ of the  $L_{\ve \nu}$. To perform this step, we use a link between iterations of ${\ve \nu}$ modulo one and the computational Diophantine approximations.
Next we use a version of binary search to find in the interval $[0,1]$ two rationals $r^{-}$ and $r^{+}$ with $r^{-}<r^{+}$, satisfying the following properties. First, the numbers $r^{-}$, $r^{+}$ are relatively close to each other, so that $r^{+}-r^{-}<\epsilon$. Second, Babai's nearest plane algorithm applied to ${\ve u}_1, {\ve u}_2, \ldots, {\ve u}_{d}$ and ${\ve c}={\ve c}(r^{+})$ finds a lattice point  ${\ve \xi}\in L_{\ve \nu}$
such that ${\ve \xi}\in \coneS$ and the same algorithm applied to ${\ve u}_1, {\ve u}_2, \ldots, {\ve u}_{d}$ and ${\ve c}={\ve c}(r^{-})$ fails to find a lattice point in $\coneS$.
This will imply that ${\ve \xi}$  satisfies conditions of Theorem \ref{length_bound}.

The algorithm is given below.

\vskip.5cm
{\bf Algorithm}
\begin{itemize}
\item[\em Input]: ${\ve \nu}$ of the form
(\ref{nu_via_p}) and rational $\epsilon\in (0,1)$.

\item[\em Output]: ${\ve \xi}\in L_{\ve \nu}$ satisfying conditions of Theorem \ref{length_bound}.

\item[\em Step 0]: Set $r^{-}:=0$, $r^{+}:=1$ and ${\ve \xi}:={\ve c}(1)$.

\item[\em Step 1]: Compute a basis ${\ve u}_1, {\ve u}_2, \ldots, {\ve u}_{d}$ of the lattice $L_{\ve \nu}$.

\item[\em Step 2]: {\bf While} $r^{+}-r^{-}>\epsilon$ {\bf do}

\begin{itemize}

\item[\em 2.1] Set $m:=(r^{-}+r^{+})/2$ and ${\ve c}:={\ve c}(m)$.

\item[\em 2.2]  Apply the Babai's algorithm for finding a
  nearby lattice point to the basis ${\ve u}_1, \ldots, {\ve u}_{d}$
  and the point ${\ve c}$. The algorithm
  returns a lattice point ${\ve \chi}\in L_{\ve \nu}$.

\item[\em 2.3] {\bf If} ${\ve \chi}\in \coneS$ {\bf then} set $r^{+}:=m$
{\bf else} set $r^{-}:=m$ {\bf end if}.

\end{itemize}

{\bf end while}.

\item[\em Step 3]:  Output vector ${\ve \xi}$.

\end{itemize}

Let us now analyze the algorithm. Clearly, Step 0 can be done in polynomial time.
In Step 1 we can compute a basis of $L_{\ve \nu}$ as follows. Consider the  matrix $G({\ve \nu})\in \Q^{d\times(d+1)}$ defined as

\bea
G({\ve \nu})=\left (
\begin{array}{ccccc}
1 & 0  & \ldots & 0 & p_1/q\\
0 & 1  & \ldots & 0 & p_2/q\\
\vdots & \vdots & \ddots &\vdots  & \vdots \\
0 & 0 &  \ldots & 1 & p_d/q
\end{array}
\right)\,
\eea
and denote by ${\ve g}_i$ its $i$th column vector. Observe that $L_{\ve \nu}=\{y_1{\ve g}_1+\ldots+y_{d+1}{\ve g}_{d+1}: y_1, \ldots, y_{d+1}\in \Z\}$.
Thus we can find a basis of $L_{\ve \nu}$ in polynomial time by Corollary 5.4.8 of \cite{GLS} (see also \cite{BuchPohst}).

The {\em while loop} at Step 2 is performing a binary search in the interval $[0,1]$ with approximation error bounded by $\epsilon$ and thus will be executed $O(l(\epsilon))$ times, where $l(\epsilon)$ is the length of the binary expansion of the rational number $\epsilon$.
The algorithm of Babai (see \cite{Babaika}), applied at Step 2.2, runs in polynomial time.
 Step 2.3 can be done in polynomial time as well.

Thus it is now enough to show that the vector ${\ve \xi}$ output at Step 3 satisfies conditions
of Theorem \ref{length_bound}.
By Theorem \ref{Babai}, we clearly have ${\ve \xi}\in B({\ve c}(r^+), 2^{d/2}\tau(L_{\ve \nu}))$.
Next, since  Babai's algorithm applied to ${\ve u}_1, \ldots, {\ve u}_{d}$ and the point
${\ve c}(r^{-})$ returns a lattice point outside of $\coneS$, we also conclude by Theorem
\ref{Babai} that $r^{-}\le 2^{d/2}\tau(L)$. The latter inequality together with
$r^{+}-r^{-}\le \epsilon$ implies then
\bea
\begin{split}
r^+\le 2^{d/2}\tau(L)+\epsilon\,.
\end{split}
\eea

Therefore the point ${\ve \xi}$ satisfies conditions  of Theorem \ref{length_bound}.

{\em Remark}.
It is easy to see that, in fact, we are solving in the above proof a problem of simultaneous Diophantine
approximation of rationals $p_1/q,\ldots, p_d/q$. Indeed, all points of the lattice $L_{\ve \nu}$ have
the form $(y_1-y_{d+1}p_1/q, \ldots, y_d-y_{d+1}p_d/q)$ with integer numbers $y_i$. It may also be
worthwhile using another standard approach to computing Diophantine approximations with bounded
denominators for a given rational vector. In this case, we construct a basis of a special lattice
$\Omega\in\Q^{d+1}$ with $\pi_{d+1}(\Omega)=L_{\ve \nu}$. For details, see the proof of Theorem 5.3.19
in \cite{GLS} or, for a more recent approach, Chapter 6 in \cite{LLL}.

\subsection{Approximation for the multiplicative strategy}
\label{section_mult}

For the rest of the paper we set $d=m+1$.
Given the multiplier vector ${\ve \lambda}$ of the form (\ref{standard_lambda}), we construct the augmented vector
\[
{\ve \nu} = ( \lambda_1, \ldots, \lambda_m, \nu )^T\in (0,1)^d,
\]
and attempt to find a vector ${\ve \xi}=t {\ve \nu} \mod 1$, $\xi_d>0$, with minimum  ratio
\bea
r({\ve \xi})=||\pi_d({\ve \xi})||/\xi_d\,.
\eea
Recall that for $Y\subset  \R^{d}$ by $\pi_{d}(Y)$ we understand the orthogonal projection of $Y$ onto the coordinate hyperplane $x_{d}=0$; we view $\pi_{d}(Y)$ as a subset of $\R^{d-1}$.

As it was remarked in Section 4, for any given common denominator $q>1$ there exist rational vectors ${\ve \nu}$ of the form (\ref{nu_via_p}) with $\tau(L_{\ve \nu})\gg 1$.
However, due to Theorem
\ref{probability}, for a typical ${\ve \nu}$ the covering radius of the lattice $L_{\ve \nu}$ is of order $q^{-1/d}$.   In the following
we show the existence of a  vector
${\ve \xi}=t {\ve \nu} \mod 1$, with ratio $r({\ve \xi})$  bounded in terms of the covering radius.
We also show the existence of a polynomial-time algorithm which computes an
approximation of that vector ${\ve \xi}$.

For $0<R<1/2$ set
\bea
a(d, R)=\frac{(1-R)((d-1)R^2-2R+1)^{1/2}-(d-1)^{1/2}R^2}{dR^2-2R+1}\,
\eea
and
\bea
r(d, R)=
\left\{
\begin{array}{ll}
(a(d,R)^{-2}-1)^{1/2}\, & \mbox{for}\; 0<R<1/2\,,\\
+\infty & \mbox{otherwise}\,.
\end{array}
\right.
\eea

We will first prove a simple geometric lemma.

\begin{lemma}
Let $0<R< 1/2$. Then
\be
\max\{r({\ve x}): {\ve x}\in B^d({\ve c}(R), R)\}=r(d, R)\,.
\label{isoperimetric_inequality}
\ee
\label{isoperimetric_lemma}
\end{lemma}
\begin{proof}
For any fixed $1-R\le y\le 1$ , the maximum
\bea \max\{r({\ve x}): {\ve x}=(x_1, \ldots, x_{d-1}, y)\in B^d({\ve c}(R), R)\}\eea
is attained at a point of the form $(x, \ldots, x, y)$. Thus we can consider only two variables, $x$ and $y$, and (\ref{isoperimetric_inequality}) reduces to solving a 2-dimensional trigonometric problem. Straightforward computation gives
\be
\max\left\{\frac{\sqrt{d-1}\;x}{y}: (x, \ldots, x, y)\in B^d({\ve c}(R), R)\right\}=r(d, R)\,.
\label{2Dtrigonometry}
\ee
\end{proof}

By (\ref{2Dtrigonometry}), we also have $r(d, R)\ll_d R$ when $0<R<1/2$.

\begin{propo}
There exists a point ${\ve \xi}=t{\ve \nu} \mod 1$, $1\le t\le q-1$,
with
\be
r({\ve \xi})\le \min\{r(d,\tau(L_{\ve \nu})), 2\sqrt{d-1}\}\,.
\label{r_upper_bound}
\ee
\label{proposition_ratio}
\end{propo}

\begin{proof}
Observe first that there is a positive integer $t_0$ such that
$1/2 \le \{t_0 \nu\} < 1$. Thus for ${\ve \xi}=t_0 {\ve \nu} \mod 1$, we have
\bea
r({\ve \xi})< 2\sqrt{d-1}\,.
\eea
This justifies the second bound in (\ref{r_upper_bound}).

Recall that the iterations $t{\ve \nu} \mod 1$ can be naturally embedded in the lattice $L_{\ve \nu}$.
Thus, it is enough to show that there exists a nonzero point ${\ve \xi}\in L_{\ve \nu}\cap [0,1)^d$ that satisfies the first inequality in (\ref{r_upper_bound}).
If $\tau(L_{\ve \nu})\ge 1/2$, the latter inequality holds by the definition of $r(d, R)$.
Suppose that $\tau(L_{\ve \nu})< 1/2$. By the definition of the covering radius there exists a point ${\ve \xi}\in L_{\ve \nu}\cap B^d({\ve c}(\tau(L_{\ve \nu})), \tau(L_{\ve \nu}))$. Since $\tau(L_{\ve \nu})< 1/2$, the point ${\ve \xi}$ is in $[0,1)^d$. The  first inequality in (\ref{r_upper_bound}) now holds by Lemma \ref{isoperimetric_lemma}.
\end{proof}

On the algorithmic side, Theorem \ref{length_bound} implies the following result.

\begin{coro}
There is a polynomial-time algorithm which, given an augmented vector ${\ve \nu}=(\lambda_1,\ldots,\lambda_{m}, \nu)$ of the form
(\ref{nu_via_p}) and any rational $\epsilon\in (0,1)$, finds a point ${\ve \xi}=t{\ve \nu} \mod 1$, $1\le t\le q-1$, with
\be
r({\ve \xi})< \min\{r(d,2^{d/2}\tau(L_{\ve \nu})+\epsilon), 2\sqrt{d-1}\}\,.
\label{r_algorithmic_upper_bound}
\ee
\label{sim_ratio}
\end{coro}
\begin{proof}
The first bound in (\ref{r_algorithmic_upper_bound}) immediately follows from Theorem \ref{length_bound} and Lemma \ref{isoperimetric_lemma}, where we take $R=2^{d/2}\tau(L_{\ve \nu})+\epsilon$.

Next, if $\nu\ge 1/2$, we have $r({\ve \nu})<2\sqrt{d-1}$, so the second bound in (\ref{r_algorithmic_upper_bound}) holds for ${\ve \xi}={\ve \nu}$.
If $0 < \nu < 1/2$, then we can take ${\ve \xi}=t_0 {\ve \nu} \mod 1$ with $t_0=\lfloor 1 / \nu \rfloor$ when $\lfloor 1 / \nu \rfloor \nu\neq 1$ and  $t_0=\lfloor 1 / \nu \rfloor-1$ otherwise.
\end{proof}

\subsection{Approximation for the additive strategy}
\label{section_add}

Now we move on to the additive strategy. As in the previous section, for a non-trivial CG-cut
(\ref{eq:CGC}) with ${\ve \lambda}$ of the form (\ref{standard_lambda}) we construct the
augmented vector ${\ve \nu}=(\lambda_1, \ldots, \lambda_m, \nu)^T$. One can easily obtain
the following bound for Problem \ref{SCP}.

\begin{propo}
There exists a point ${\ve \xi}=t{\ve \nu} \mod 1$, $1\le t\le q-1$,
with
\be
N({\ve \xi})\le (1+\sqrt{d})\tau(L_{\ve \nu})\,.
\label{bound_for_xi}
\ee
Furthermore,
\be
\{t\nu\}>0\;\mbox{whenever}\; \tau(L_{\ve \nu})<1/2\,.
\label{positive_rounding}
\ee

\label{theoretical_bound}
\end{propo}

\begin{proof}

Observe that the set $B^d({\ve c}(0), (1+\sqrt{d})\tau(L_{\ve \nu}))\cap \coneS$ contains the ball $B^d({\ve c}(\tau(L_{\ve \nu})), \tau(L_{\ve \nu}))$.
By the definition of the covering radius there exists a point ${\ve \chi}\in L_{\ve \nu}\cap B^d({\ve c}(\tau(L_{\ve \nu})), \tau(L_{\ve \nu}))$, so that $N({\ve \chi})\le (1+\sqrt{d})\tau(L_{\ve \nu})$.
If ${\ve \chi}\in \Z^d$ then we may assume without loss of generality that ${\ve \chi}={\ve c}(1)$. Thus in this case we can take ${\ve \xi}={\ve \nu}$.
Otherwise, since ${\ve \chi}\in \coneS\setminus \Z^d$, we have $0<N({\ve \chi}\mod 1)\le N({\ve \chi})$.
Thus, the point ${\ve \xi}={\ve \chi}\mod 1$ satisfies condition (\ref{bound_for_xi}).

Suppose now that $\tau(L_{\ve \nu})<1/2$. Then for all sufficiently small $\epsilon>0$ the ball $B^d({\ve c}((\tau(L_{\ve \nu})+\epsilon)), \tau(L_{\ve \nu}))$ contains
a  point of the set $L_{\ve \nu}\cap (0,1)^d$. Since $L_{\ve \nu}$ is a discrete set, we conclude that there exists a point ${\ve \xi}\in L_{\ve \nu}\cap B^d({\ve c}(\tau(L_{\ve \nu})), \tau(L_{\ve \nu}))\cap (0,1)^d$.
This point clearly satisfies (\ref{positive_rounding}).
\end{proof}

On the other hand, Theorem  \ref{length_bound} implies the following

\begin{coro}
There is a polynomial-time algorithm which, given an augmented vector
${\ve \nu}=(\lambda_1,\ldots,\lambda_{m}, \nu)$ of the form
(\ref{nu_via_p}) and any rational $\delta\in (0,1)$, finds a point
${\ve \xi}=t{\ve \nu} \mod 1$, $1\le t\le q-1$, with
\be
N({\ve \xi})< (1+\sqrt{d})2^{d/2}\tau(L_{\ve \nu})+\delta\,.
\label{alg_bound_for_xi_add}
\ee
Furthermore,
\be
\{t\nu\}>0\;\mbox{whenever}\; \tau(L_{\ve \nu})< 2^{-d/2-1}(1-\delta/\lceil \sqrt{d}\rceil)\,.
\label{positive_rounding_alg}
\ee
\label{sim}
\end{coro}

\begin{proof}

By Theorem \ref{length_bound}, given ${\ve \nu}=(\lambda_1,\ldots,\lambda_{m}, \nu)$ and $\epsilon=\delta/\lceil \sqrt{d}\rceil\in (0,1)$ we can compute in polynomial time a point ${\ve \xi}\in L_{\ve \nu}\cap \coneS$ such that ${\ve \xi}\in  B({\ve c}(r), 2^{d/2}\tau(L_{\ve \nu}))$ with $0<r\le 2^{d/2}\tau(L_{\ve \nu})+\epsilon$. Thus $N({\ve \xi})\le N({\ve c}(r))+2^{d/2}\tau(L)$ and, consequently,
\bea
N({\ve \xi})\le (2^{d/2}\tau(L)+\epsilon)\sqrt{d}+2^{d/2}\tau(L)
\le (1+\sqrt{d})2^{d/2}\tau(L)+\delta.
\eea
Therefore the point ${\ve \xi}$ satisfies the inequality (\ref{alg_bound_for_xi_add}).

Suppose now that $\tau(L_{\ve \nu})<2^{-d/2-1}(1-\delta/\lceil \sqrt{d}\rceil)$. Clearly, $\{t\nu\}=\{\xi_d\}$, so it is enough to show that  $\xi_d\in (0,1)$. Since ${\ve \xi}\in \coneS$,  the number $\xi_d$ is positive. On the other hand, we have
\bea
\xi_d\le r + 2^{d/2}\tau(L) \le 2^{d/2+1}\tau(L)+\delta/\lceil \sqrt{d}\rceil<1\,.
\eea
\end{proof}

\subsection{Approximation error}

As it is shown in Sections \ref{section_mult} and \ref{section_add}, the computed approximations of the optimal values of $r({\ve \xi})$ and $N({\ve \xi})$ are bounded in terms of the covering radius and thus are small for a typical augmented vector. We conjecture that the iterated CG-cuts found by the algorithms obtained in Corollaries \ref{sim_ratio} and \ref{sim} solve problems (\ref{SM}) and (\ref{SCP}), respectively, with the multiplicative approximation error $2^{O(d)}$.
In this section we prove the second conjecture for the special case $\tau(L_{\ve \nu})\ll q^{-1/d}$, where $\ll$ is the Vinogradov symbol.

Let ${\ve \nu}=(\lambda_1,\ldots,\lambda_{m}, \nu)$ be a vector of the form
(\ref{nu_via_p}) and let  $\delta\in (0,1)\cap \Q$.
We will denote by $m_{add}=m_{add}({\ve \nu})$ the value of the minimum in (\ref{SCP}), that is
\bea
m_{add}({\ve \nu})=\min \left\{ N(t{\ve \nu}\mod 1): \;
t=1,\ldots,q-1, \; \{t \nu\}>0 \right\}.
\eea
We will also denote by ${\ve \xi}_{add}={\ve \xi}_{add}({\ve \nu}, \delta)$ the output vector of the algorithm obtained in Corollary \ref{sim}.

\begin{propo} Let ${\ve \nu}$ be a vector of the form
(\ref{nu_via_p}) with common denominator $q$. Then
\be
\frac{N({\ve \xi}_{add}({\ve \nu}, 1/q))}{m_{add}({\ve \nu})}< 2^{3d/2-1} (1+\sqrt{d})\tau(L_{\ve \nu})^dq+1\,.
\label{approx_error_inequ}
\ee
\label{approx_error}
\end{propo}

\begin{proof}
Recall that $L_{\ve \nu} = \Z^d + \Z {\ve \nu}$. Therefore for the first successive minimum $\lambda_1=\lambda_1(L_{\ve \nu})$ we obtain the inequalities
$1/q\le \lambda_1 \le m_{add}({\ve \nu})$. Together with (\ref{alg_bound_for_xi_add}) this observation implies the inequality
\be
\frac{N({\ve \xi}_{add}({\ve \nu}, 1/q))}{m_{add}({\ve \nu})}< \frac{(1+\sqrt{d})2^{d/2}\tau(L_{\ve \nu})}{\lambda_1}+1\,.
\label{via_lambda_1}
\ee
By Minkowski's Second theorem for spheres, $1/q=\det(L_{\ve \nu})\le \lambda_1\lambda_2\cdots\lambda_d$ and hence
\be
\lambda_1\ge \frac{1}{q \lambda_d^{d-1}}\,.
\label{Minkowski_spheres}
\ee
Next, by Jarnik's inequalities (cf.~\cite[p.~99, p.~106]{GrLek})) we have $\lambda_d\le 2 \tau(L_{\ve \nu})$. Consequently, by (\ref{Minkowski_spheres})
\be
\lambda_1\ge \frac{1}{2^{d-1}q \tau(L_{\ve \nu})^{d-1}}\,.
\label{lambda_1_via_tau}
\ee
Combining (\ref{via_lambda_1}) and (\ref{lambda_1_via_tau}), we obtain the inequality (\ref{approx_error_inequ}).

\end{proof}

Proposition \ref{approx_error} immediately implies the inequality $N({\ve \xi}_{add}({\ve \nu}, 1/q))< 2^{O(d)}m_{add}({\ve \nu})$,
provided $\tau(L_{\ve \nu})\ll q^{-1/d}$.

A natural step towards establishing both conjectures would be to show that the approximation error is independent of the common denominator $q$.
In this light, Proposition \ref{approx_error}, together with Theorem \ref{bound_for_P} imply that for a {\em typical} input vector ${\ve \nu}$ the problem $(\ref{SCP})$ can be approximated with the multiplicative approximation error that only depends on $d$.

\section{Concluding Remarks} \label{se:end}

Although Chv\'atal-Gomory cuts have been around for over 50 years and have been studied in depth,
many important questions about them remain unanswered. We have studied the behavior of the iterated CG-cuts for a randomly chosen
augmented vector and have shown the existence of a polynomial-time algorithm that computes approximations for the problems \ref{SM} and \ref{SCP}.
For computed approximations the values of $r({\ve \xi})$ and $N({\ve \xi})$ are bounded in terms of the covering radius and thus are small for a typical augmented vector. On the other
hand, we do not know the precise approximation ratio that this algorithm yields. Nor do we know
the precise approximability (or inapproximability) status of the problems \ref{SM} and \ref{SCP}.
Moreover, our algorithm seems at present of mainly theoretical interest, though this may change
in the near future, given the intensive recent work on algorithms for integer lattices (see the
survey \cite{HSP}).

We also remark that the strategy presented in this paper is designed to optimize individual CG-cuts
only. On the other hand, since the work of Balas {\em et al.}\ \cite{BCCN96}, most integer programmers
prefer to work with collections of cutting planes rather than individual ones. (Specifically, given
a fractional simplex tableau, one can generate one GF-cut for each fractional variable, and add all
such GF-cuts to the LP relaxation.) It is not clear that optimising each CG-cut in a collection will
improve the effectiveness of the entire collection. Indeed, in our computational experiments, we often
observed that different CG-cuts led to the same strengthened iterated CG-cut, so that a large collection
of weak CG-cuts was converted into a small collection of strong ones. This suggests that a suitable
topic for future research might be the simultaneous optimization of a collection of CG-cuts.
A method for strengthening a collection of Gomory {\em mixed-integer}\/ cuts, rather than
GF-cuts, was presented in \cite{ACL05}.

\end{document}